\documentclass[11pt]{amsart}
\usepackage{amsmath,amssymb,amsthm,upref,graphicx,mathrsfs, enumerate}

\usepackage{color,xcolor}
\usepackage[
  colorlinks=true,
  linkcolor=blue,
  citecolor=blue,
  urlcolor=blue,
  pagebackref]{hyperref}

\numberwithin{equation}{section}
\allowdisplaybreaks

\newtheorem{theorem}{Theorem}[section]
\newtheorem{prop}[theorem]{Proposition}
\newtheorem{lemma}[theorem]{Lemma}

\theoremstyle{definition}

\newcommand{\Mm}{\mathcal{M}}

\newcommand{\Rn}{\mathbb{R}^n}

\newcommand{\nw}{\nu_{\vec{w}}}

\newcommand{\f}{\frac}
\newcommand{\vc}{\infty}

\def\Xint#1{\mathchoice
{\XXint\displaystyle\textstyle{#1}}%
{\XXint\textstyle\scriptstyle{#1}}%
{\XXint\scriptstyle\scriptscriptstyle{#1}}%
{\XXint\scriptscriptstyle\scriptscriptstyle{#1}}%
\!\int}
\def\XXint#1#2#3{{\setbox0=\hbox{$#1{#2#3}{\int}$}
\vcenter{\hbox{$#2#3$}}\kern-.5\wd0}}

\def\dashint{\Xint-}

\let \a=\alpha

\let \la=\lambda
\let \o=\omega

\begin{document}

\title[Weighted bounds for multilinear square functions]
  {Weighted bounds for multilinear square functions}

\authors

\author[T. A. Bui]{The Anh Bui}
\address{The Anh Bui \\
Department of Mathematics, Macquarie University, Ryde 2109 NSW, Australia}
\email{the.bui@mq.edu.au}

\author[M. Hormozi]{Mahdi Hormozi}
\address{Mahdi Hormozi \\
Department of Mathematical Sciences, Division of Mathematics,
University of Gothenburg, Gothenburg 41296, Sweden}
\email{hormozi@chalmers.se}


\arraycolsep=1pt

\begin{abstract} Let $\vec{P}=(p_1,\dotsc,p_m)$ with $1<p_1,\dotsc,p_m<\infty$, $1/p_1+\dotsb+1/p_m=1/p$ and $\vec{w}=(w_1,\dotsc,w_m)\in A_{\vec{P}}$. In this paper, we investigate the weighted bounds with dependence on aperture $\alpha$ for multilinear square functions  $S_{\alpha,\psi}(\vec{f})$. We show that
$$
\|S_{\alpha,\psi}(\vec{f})\|_{L^p(\nu_{\vec{w}})} \leq C_{n,m,\psi,\vec{P}}~ \alpha^{mn}[\vec{w}]_{A_{\vec{P}}}^{\max(\frac{1}{2},\tfrac{p_1'}{p},\dotsc,\tfrac{p_m'}{p})} \prod_{i=1}^m \|f_i\|_{L^{p_i}(w_i)}.
  $$
This result extends  the result in the linear case which was obtained by Lerner in 2014. Our proof is based on the local mean oscillation technique presented firstly to find the weighted bounds for Calder\'on--Zygmund operators. This method helps us avoiding intrinsic square functions in the proof of our main result.
\end{abstract}

\maketitle
\section{Introduction}
\par The problem of the optimal quantitative estimates for the $L^p(w)$ norm of a given operator $T$ in terms of the $A_p$ constant of the weight $w$ has been very challenging and interesting in the last decades.
\par First, the problem for the Hardy--Littlewood maximal operator was solved by S. Buckley \cite{Bu} who proved

\begin{equation}\label{Buckley_thm}
  \|M\|_{L^p(w)} \leq C_{p} \,[w]_{A_p}^{\frac{1}{p-1}},
\end{equation}
where $C_p$ is a dimensional constant. We say that \eqref{Buckley_thm} is a sharp estimate since the exponent $1/(p-1)$ cannot be replaced by a smaller one.
\par However, for singular integral operators the question was much more complicated. In 2012, T. Hyt\"onen \cite{Hyt1} proved the so-called $A_2$ theorem, which asserted that the sharp dependence of the $L^2(w)$ norm of a Calder\'on--Zygmund operator on the $A_2$ constant of the weight $w$ was linear. More precisely,
\begin{equation}
\label{eq:Ap_sharp_CZO}
  \|T\|_{L^p(w)} \leq C_{T,n,p} [w]_{A_p}^{\max{\left(1,\frac{1}{p-1}\right)}}, \,\,  1<p<\infty.
\end{equation}

Shortly after that, A.K. Lerner gave a much simpler proof \cite{Ler4} of the $A_2$ theorem proving that every Calder\'on--Zygmund operator is bounded from above by a supremum of sparse operators. Namely, if $X$ is a Banach function space, then
\begin{equation}
\label{eq:domination_CZO}
  \|T(f)\|_X \leq C \sup_{\mathscr{D},\mathcal{S}} \| \mathcal{A}_{\mathscr{D},\mathcal{S}}(f)\|_X,
\end{equation}
where the supremum is taken over arbitrary dyadic grids $\mathscr{D}$ and sparse families $\mathcal{S}\in\mathscr{D}$, and
\begin{equation*}
  \mathcal{A}_{\mathscr{D},\mathcal{S}}(f)= \sum_{Q\in \mathcal{S}}\Big(\dashint_Q f\Big)\chi_Q.
\end{equation*}
The interested readers can consult \cite{Hyt2} for a survey on the history of the proof.
\par The versatility of Lerner's techniques is reflected in the extension of \eqref{eq:domination_CZO} and the $A_2$ theorem to multilinear Calder\'on--Zygmund operators in \cite{DLP}. Later on, Li, Moen and Sun in \cite{LMS} proved the corresponding sharp weighted $A_{\vec P}$ bounds for multilinear sparse operators. In other words, if $1<p_1,\dotsc,p_m<\infty$ with $\tfrac{1}{p_1}+\dotsb+\tfrac{1}{p_m}=\tfrac{1}{p}$ and $\vec{w}\in A_{\vec{P}}$, then
\begin{equation}
\label{eq:LMS_sparse}
          \|\mathcal{A}_{\mathscr{D},S}(\vec{f})\|_{L^p(\nu_{\vec{w}})} \lesssim [\vec{w}]_{A_{\vec{P}}}^{\max(1,\tfrac{p'_1}{p},\dotsc,\tfrac{p'_m}{p})} \prod_{i=1}^m \|f_i\|_{L^{p_i}(w_i)},
\end{equation}
where $\mathcal{A}_{\mathscr{D},S}$ denotes the multilinear sparse operators
\begin{equation*}
  \mathcal{A}_{\mathscr{D},S}(\vec{f})(x)=\sum_{Q} \left(\prod_{i=1}^m (f_i)_{Q}\right) \chi_{Q}(x),
\end{equation*}
and the other notation is explained in Section 2.
The readers are referred to \cite{CR, LMS} to observe that from \eqref{eq:LMS_sparse}, we can derive the multilinear $A_{\vec P}$ theorem for $1/m<p<\infty$. More precisely, if $T$ is a multilinear Calder\'on--Zygmund operator, $1<p_1,\dotsc,p_m<\infty$, $\tfrac{1}{p_1}+\dotsb+\tfrac{1}{p_m}=\tfrac{1}{p}$ and $\vec{w}=(w_1,\dotsc,w_m)\in A_{\vec{P}}$, then
  \begin{equation}\label{eq:LMS_CZO}
    \|T(\vec{f})\|_{L^p(\nu_{\vec{w}})} \leq C_{n,m,\vec{P},T} [\vec{w}]_{A_{\vec{P}}}^{\max(1,\tfrac{p_1'}{p},\dotsc,\tfrac{p_m'}{p})} \prod_{i=1}^m \|f_i\|_{L^{p_i}(w_i)}.
  \end{equation}
For further details on the theory of multilinear Calder\'on--Zygmund operators, we refer to \cite{G, GT1} and the references therein.

Let $S_{\alpha,\phi}$ be the square function defined by means of the cone $\Gamma_\alpha$ in $\mathbb{R}^{n+1}_{+}$ of aperture $\alpha>1$, and a standard kernel $\phi$ as follows

$$
S_{\alpha, \phi}(f)(x)=\Big(\int_{\Gamma_\alpha(x)}|f\star \phi_t(y)|^2\f{dydt}{t^{n+1}}\Big)^{1/2},
$$
 where $\phi_t(x)=t^{-n} \phi(x/t)$ and $\star$ refers to convolution operation of two functions. In \cite{Ler6}, Lerner by applying \textit{intrinsic square functions}, introduced in \cite{Wil}, proved sharp weighted norm inequalities for $S_{\alpha, \phi}(f)$ . Later on, Lerner himself improved the result--- in the sense of determination of sharp dependence on $\alpha$  --- in \cite{Ler5} by using the local mean oscillation formula. More precisely,
%
\begin{equation}
\label{eq:sharp_lsq}
  \|S_{\alpha,\phi}\|_{L^p(w)} \lesssim \alpha^n [w]_{A_p}^{\max{\left(\frac{1}{2},\frac{1}{p-1}\right)}}, \,\,  1<p<\infty.
\end{equation}

Motivated by these works, the main aim of this paper is to investigate the weighted bounds for certain multilinear square functions. Let us recall the definition of multilinear square functions considered in this paper.

For any $t\in (0,\infty)$, let $\psi(x,\vec{y}):=K_t(x, y_1, \dotsc, y_m)$ be a locally integrable function defined away from the diagonal $x =y_1=\dotsc =y_m$ in $\mathbb{R}^{n\times (m+1)}$. We assume that there are positive constants $\delta$ and $ A$ so that the following conditions hold.

\begin{enumerate}[Size condition:]
 \item \begin{equation}\label{eq-sizecondition}
 |\psi(x,\vec{y})|\leq \f{A}{(1+|x-y_1|+\dotsb+|x-y_m|)^{mn+\delta}}.
  \end{equation}
  \end{enumerate}

\begin{enumerate}[Smoothness condition:]
 \item There exists $\gamma>0$ so that
 \begin{equation}\label{eq-smoothcondition1}
 |\psi(x,\vec{y})-\psi(x+h,\vec{y})|\leq \f{A|h|^{\gamma}}{(1+|x-y_1|+\dotsb+|x-y_m|)^{mn+\delta+\gamma}},
  \end{equation}
 whenever $|h|<\f{1}{2}\max_j|x-y_j|$,  and
\begin{equation}\label{eq-smoothcondition2}
 |\psi(x,y_1,\dotsc, y_i,\dotsc, y_m)-\psi(x,y_1,\dotsc, y_i+h,\dotsc, y_m)|\leq \f{A|h|^{\gamma}}{(1+|x-y_1|+\dotsb+|x-y_m|)^{mn+\delta+\gamma}},
  \end{equation}
  whenever $|h|<\f{1}{2}|x-y_i|$ for $i\in \{1,\dotsc,m\}$.
\end{enumerate}
For $\vec{f}=(f_1,\dotsc,f_m)\in \mathcal{S}(\mathbb{R}^n)\times\dotsb\times
\mathcal{S}(\mathbb{R}^n)$ and $x\notin \bigcap_{j=1}^m {\rm supp}\, f_j$ we define
$$
\psi_t(\vec{f})(x)=\f{1}{t^{mn}}\int_{(\mathbb{R}^n)^m}\psi\Big(\f{x}{t},\f{y_1}{t},\dotsc, \f{y_m}{t}\Big)\prod_{j=1}^m f_j(y_j)dy_j.
$$
For $\lambda>2m, \alpha>0$, the multilinear square functions $g^*_{\lambda, \psi}$ and $S_{\psi,\alpha}$ associated to $\psi(x,\vec{y})$ are defined by
$$
\begin{aligned}
g^*_{\lambda, \psi}(\vec{f})(x)&=\Big(\int_{\mathbb{R}^{n+1}_+}\Big(\f{t}{t+|x-y|}\Big)^{n\lambda}|\psi_t(\vec{f})(y)|^2\f{dydt}{t^{n+1}}\Big)^{1/2},
\end{aligned}
$$
and
$$
\begin{aligned}
S_{\alpha, \psi}(\vec{f})(x)&=\Big(\int_{\Gamma_\alpha(x)}|\psi_t(\vec{f})(y)|^2\f{dydt}{t^{n+1}}\Big)^{1/2},
\end{aligned}
$$
where $\Gamma_\alpha(x)=\{(y,t)\in \mathbb{R}^{n+1}_+: |x-y|<\alpha t\}$.

These two mutilinear square functions were introduced and investigated in \cite{CXY, SXY, XY}. The study on the multilinear square functions has important applications in PDEs and other fields. For further details on the theory of multilinear square functions and their applications, we refer to \cite{CDM, CMM, CM, DJ, FJK1, FJK2, FJK3, H, CHO, XY, CXY, H} and the references therein.

\medskip

In this paper, we assume that there exist some
$1\leq p_1,\dotsc,p_m \leq \infty$
and some $0<p<\infty$
with $\frac{1}{p}=\frac{1}{p_1}+\dotsb+\frac{1}{p_m}$,
such that $g^*_{\lambda, \psi}$ maps continuously
$L^{p_1}(\Rn)\times\dotsb\times L^{p_m}(\Rn)\to L^{p}(\Rn)$. Under this condition, it was proved in \cite{XY} (see also \cite{SXY}) that $g^*_{\lambda, \psi}$ maps continuously $L^{1}(\Rn)\times\dotsb\times L^{1}(\Rn)\to L^{1/m,\vc}(\Rn)$ provided $\lambda>2m$. Moreover, since $S_{\alpha,\psi}$ is dominated by $g^*_{\lambda, \psi}$, we also get that $S_{\alpha, \psi}$ maps continuously $L^{1}(\Rn)\times\dotsb\times L^{1}(\Rn)\to L^{1/m,\vc}(\Rn)$. {The next theorem gives the  weighted bounds depending on $\alpha$ for multilinear square functions $S_{\alpha,\psi}(\vec{f})$.}
\begin{theorem}
  \label{thm:sharp}
     Let $\vec{P}=(p_1,\dotsc,p_m)$ with $1<p_1,\dotsc,p_m<\infty$ and $1/p_1+\dotsb+1/p_m=1/p$. Let $\alpha \geq 1$. If $\vec{w}=(w_1,\dotsc,w_m)\in A_{\vec{P}}$, then
  \begin{equation}
  \label{eq:sharpGCZO}
    \|S_{\alpha,\psi}(\vec{f})\|_{L^p(\nu_{\vec{w}})} \leq C_{n,m,\psi,\vec{P}} \alpha^{mn}[\vec{w}]_{A_{\vec{P}}}^{\max(\frac{1}{2},\tfrac{p_1'}{p},\dotsc,\tfrac{p_m'}{p})} \prod_{i=1}^m \|f_i\|_{L^{p_i}(w_i)}.
  \end{equation}
\end{theorem}

For the weighted bounds for $g^*_{\lambda,\psi}$ functions, we have the following result.
\begin{theorem}
  \label{thm:sharp2}
     Let $\lambda> 2m$, $\vec{P}=(p_1,\dotsc,p_m)$ with $1<p_1,\dotsc,p_m<\infty$ and $1/p_1+\dotsb+1/p_m=1/p$. If $\vec{w}=(w_1,\dotsc,w_m)\in A_{\vec{P}}$, then
  \begin{equation}
  \label{eq:sharp2}
    \|g^*_{\lambda, \psi}(\vec{f})\|_{L^p(\nu_{\vec{w}})} \leq C_{n,m,\psi,\vec{P}}[\vec{w}]_{A_{\vec{P}}}^{\max(\frac{1}{2},\tfrac{p_1'}{p},\dotsc,\tfrac{p_m'}{p})} \prod_{i=1}^m \|f_i\|_{L^{p_i}(w_i)}.
  \end{equation}
\end{theorem}

We would like to point out that in the linear case, Theorem \ref{thm:sharp} gives  the sharp weighted bounds with sharp dependence on $\alpha$ whereas Theorem \ref{thm:sharp2} provides  sharp weighted bounds for square functions. See for example \cite{Ler4, Ler5}. Although our conjecture is that these bounds are sharp, we couldn't prove this and leave it as an open problem.

\par The outline of this paper will be as follows. In Section~\ref{Sect:Preliminaries} we establish the notation that we will follow as well as some background which will be helpful in the sequel. Also, the weighted estimates of the operators $\mathcal{A}^\gamma_{\mathscr{D},S}$, which have key roles in the proof of the main result of this paper, will be obtained. In Section~\ref{Sect:3}, we study weak $(p,p)$ estimates for square functions. Finally, Section~\ref{Sect:4} contains the proofs of the main results i.e. Theorem \ref{thm:sharp}, Theorem \ref{thm:sharp2} and Theorem \ref{Thm:LMS_sparse_Ap_thm}.
\par Throughout this paper $A\lesssim B$ will denote $A\leq C B$, where $C$ will denote a positive constant independent of the weight which may change from one line to other.

\bigskip

\textbf{Acknowledgement.} The first author was supported by Australian Research Council. The authors are indebted to Jos\'e M. Conde-Alonso and Guillermo Rey for bringing their attention to \cite{CR, LN} to remove the restriction $p\geq 2$ in Theorem \ref{thm:sharp} and Theorem \ref{thm:sharp2} in the earlier version of the paper. Also, the authors are grateful to Professor Hjalmar Rosengren for several valuable comments which have improved the presentation of the paper.
\section{Preliminaries}
\label{Sect:Preliminaries}

\subsection{Multiple weight theory}
For a general account on multiple weights and related results we refer the interested reader to \cite{LOPTT}. In this section we briefly introduce some definitions and results that we will need.
\par Consider $m$ weights $w_1,\dotsc,w_m$ and denote $\overrightarrow{w}=(w_1,\dotsc,w_m)$.  Also let $1<p_1,\dotsc,p_m<\infty$ and $p$ be numbers such that $\frac{1}{p}=\frac{1}{p_1}+\dotsb+\frac{1}{p_m}$ and denote $\overrightarrow P = (p_1,\dotsc, p_m)$. Set

  $$\nu_{\vec w}:=\prod_{i=1}^m w_i^{\frac{p}{p_i}}.$$
We say that $\vec{w}$ satisfies the $A_{\vec{P}}$ condition if
\begin{equation}\label{eq:multiap_LOPTT}
  [\vec{w}]_{A_{\vec{P}}}:=\sup_{Q}\Big(\frac{1}{|Q|}\int_Q\nu_{\vec w}\Big)\prod_{j=1}^m\Big(\frac{1}{|Q|}\int_Q w_j^{1-p'_j}\Big)^{p/p'_j}<\infty.
\end{equation}
When $p_j=1$, $\Big(\frac{1}{|Q|}\int_Q w_j^{1-p'_j}\Big)^{p/p'_j}$ is understood as $\displaystyle(\inf_Qw_j)^{-p}$. This condition, introduced in \cite{LOPTT}, was shown to characterize the classes of weights for which the multilinear maximal function $\Mm$ is bounded from $L^{p_1}(w_1)\times\dotsb\times L^{p_m}(w_m)$ into $L^p(\nu_{\vec{w}})$ (see \cite[Thm. 3.7]{LOPTT}).

\subsection{Dyadic grids and sparse families}

  For the notion of  {\it general dyadic grid} ${\mathscr{D}}$ we refer to previous papers (e.g. \cite{Ler3} and \cite{Hyt2}).  The collection $\{Q\}$ is called a {\it sparse family} of cubes if there are pairwise disjoint subsets $E(Q) \subset Q$ with $|Q|\le 2|E_Q|$.

\par Let $\sigma\in A_\infty$ where $A_\infty$ is the class of Muckenhoupt weights. We now define the dyadic maximal function with respect to $\sigma$
$$
M_{\sigma}^{\mathscr{D}}(f)(x)= \sup_{  Q\ni x, Q\in  \mathscr{D} } \frac{1}{\sigma(Q)} \int_{Q} |f|\sigma.
$$
By different proofs (see e.g \cite{Moen}), it is well-known that
\begin{equation}\label{bbs}
\| M_{\sigma}^{\mathscr{D}}f\|_{L^p(\sigma) } \leq p' \| f\|_{L^p(\sigma) },~~~~~~~~~~1<p<\infty.
\end{equation}

Finally, given a sparse family $\mathcal{S}$ over a dyadic grid $\mathscr{D}$ and $\gamma\geq 1$, a \textit{multilinear sparse operator} is an averaging operator over $\mathcal{S}$ of the form

  \begin{equation*}
    \mathcal{A}^\gamma_{\mathscr{D},\mathcal{S}}(\vec{f})(x)=\Big[\sum_{Q\in \mathcal{S}} \Big(\prod_{i=1}^m (f_i)_{Q}\Big)^\gamma \chi_{Q}(x)\Big]^{1/\gamma}.
  \end{equation*}

These operators verify the following multilinear $A_p$ theorem that was proved in \cite{DLP} and \cite[Thm. 3.2.]{LMS} for $\gamma=1$. In Section \ref{Sec4}, we prove the similar estimate for $\gamma\geq 1$.

  \begin{theorem}
  \label{Thm:LMS_sparse_Ap_thm}
    Suppose that $1<p_1,\dotsc,p_m<\infty$ with $\tfrac{1}{p_1}+\dotsb+\tfrac{1}{p_m}=\tfrac{1}{p}$ and $\vec{w}\in A_{\vec{P}}$. Then
        \begin{equation*}
          \|\mathcal{A}^\gamma_{\mathscr{D},\mathcal{S}}(\vec{f})\|_{L^p(\nu_{\vec{w}})} \lesssim [\vec{w}]_{A_{\vec{P}}}^{\max(\f{1}{\gamma},\tfrac{p_1'}{p},\dotsc,\tfrac{p_m'}{p})} \prod_{i=1}^m \|f_i\|_{L^{p_i}(w_i)}.
        \end{equation*}
  \end{theorem}

\subsection{A local mean oscillation formula}

The key ingredient to prove our main results is  Lerner's local oscillation formula from \cite{Ler3}. We will need to introduce the following notions to understand his result.

\par By a median value of a measurable function $f$ on a set $Q$ we mean a possibly nonunique, real number $m_f(Q)$ such that
$$\max\big(|\{x\in Q: f(x)>m_f(Q)\}|,|\{x\in Q: f(x)<m_f(Q)\}|\big)\le |Q|/2.$$

The decreasing rearrangement of a measurable function $f$ on ${\mathbb R}^n$ is defined by
$$f^*(t)=\inf\{\a>0:|\{x\in {\mathbb R}^n:|f(x)|>\a\}|<t\}\quad(0<t<\infty).$$
The local mean oscillation of $f$ is
$$\o_{\la}(f;Q)=\inf_{c\in {\mathbb R}}
\big((f-c)\chi_{Q}\big)^*\big(\la|Q|\big)\quad(0<\la<1).$$
Observe that it follows from the definitions that
\begin{equation}\label{pro1}
|m_f(Q)|\le (f\chi_Q)^*(|Q|/2).
\end{equation}
Given a cube $Q_0$, the dyadic local sharp maximal function $m^{\#,d}_{\la;Q_0}f$ is defined by
$$m^{\#,d}_{\la;Q_0}f(x)=\sup_{x\in Q'\in {\mathcal D}(Q_0)}\o_{\la}(f;Q').$$
The following theorem was proved by Hyt\"onen \cite[Theorem 2.3]{Hyt2} in order to improve  Lerner’s formula given in \cite{Ler3} by getting rid of the local sharp maximal function.
\begin{theorem}\label{decom1}
Let $f$ be a measurable function on ${\mathbb R}^n$ and let $Q_0$ be a fixed cube. Then there exists a
(possibly empty) sparse family $\mathcal{S}$ of cubes $Q\in {\mathcal D}(Q_0)$ such that for a.e. $x\in Q_0$,
\begin{equation}\label{eq:LMOD_Lerner}
  |f(x)-m_f(Q_0)|\le 2\sum_{Q \in \mathcal{S}} \o_{\frac{1}{2^{n+2}}}(f;Q)\chi_{Q}(x).
\end{equation}
\end{theorem}

\section{Weak $(p,p)$ estimate for square functions}\label{Sect:3}
For a measurable function $F\in \mathbb{R}^{n+1}_+$, we define
$$
S_\alpha(F)(x)=\Big(\int_{\Gamma_\alpha(x)}|F(y,t)|^2\f{dydt}{t^{n+1}}\Big)^{1/2},
$$
where $\Gamma_\alpha(x)=\{(y,t)\in \mathbb{R}^{n+1}_+: |x-y|<\alpha t\}$. We prove the following result on weak type $(p,p)$ estimate for $S_\alpha$.
\begin{lemma}\label{weakestimate}
Let $\alpha\geq 1$. Then for $0<p<2$ there exists $c_p$ so that
$$
\|S_\alpha(F)\|_{L^{p,\vc}}\leq c_p \alpha^{n/p}\|S_1(F)\|_{L^{p,\vc}}.
$$
\end{lemma}
\begin{proof}
Note that the case $p=1$ was proved in \cite{Ler5}. We now adapt the argument in   \cite{Ler5} to our present situation.

For $\lambda>0$ we set
$$
\Omega_\lambda=\{x: S_1(F)(x)>\lambda\} \ \ \text{and} \ \ \ U_\lambda=\{x: M\chi_{\Omega_\lambda}(x)>1/(2\alpha)^n\},
$$
where $M$ is the Hardy-Littlewood maximal function. Then by \cite[p. 315]{T}, we have
$$
\int_{\mathbb{R}^n \backslash U_\lambda} S_\alpha(F)(x)^2dx \leq 2\alpha^n \int_{\mathbb{R}^n \backslash \Omega_\lambda} S_1(F)(x)^2dx.
$$

This in combination with the weak type $(1,1)$ estimates of $M$ and Chebyshev's inequality implies that
$$
\begin{aligned}
\left|\left\{x: S_\alpha(F)(x)>\lambda\right\}\right|&\leq |U_\lambda| +\left|\left\{x\in \mathbb{R}^n \backslash U_\lambda: S_\alpha(F)(x)>\lambda\right\}\right|\\
&\leq c_n\alpha^n\left|\left\{x: S_1(F)(x)>\lambda\right\}\right| +\f{1}{\lambda^2}\int_{\mathbb{R}^n \backslash U_\lambda} S_\alpha(F)(x)^2dx\\
&\leq c_n\alpha^n\left|\left\{x: S_1(F)(x)>\lambda\right\}\right| +\f{2\alpha^n}{\lambda^2}\int_{\mathbb{R}^n \backslash \Omega_\lambda} S_1(F)(x)^2dx.
\end{aligned}
$$
On the other hand, we have
$$
\begin{aligned}
\f{2\alpha^n}{\lambda^2}\int_{\mathbb{R}^n \backslash \Omega_\lambda} S_1(F)(x)^2dx&\leq \f{4\alpha^n}{\lambda^2}\int_{0}^\lambda t|\{x: S_1(F)(x)>t\}|dt\\
&\leq \f{4\alpha^n}{\lambda^2}\|S_1(F)\|_{L^{p,\vc}}^p\int_{0}^\lambda t^{1-p}dt\\
&\leq c_p \f{\alpha^n}{\lambda^p}\|S_1(F)\|_{L^{p,\vc}}^p.
\end{aligned}
$$
Therefore,
$$
\begin{aligned}
\lambda^p\left|\left\{x: S_\alpha(F)(x)>\lambda\right\}\right| &\leq c_n\alpha^n [\lambda^p\left|\left\{x: S_1(F)(x)>\lambda\right\}\right| + \|S_1(F)\|_{L^{p,\vc}}^p],
\end{aligned}
$$
which implies that
$$
\|S_\alpha(F)\|_{L^{p,\vc}}\leq c_p\alpha^{n/p}\|S_1(F)\|_{L^{p,\vc}}.
$$
This completes our proof.
\end{proof}
\section{Proof of main results}\label{Sec4}
\label{Sect:4}
\begin{proof}[Proof of Theorem~\ref{Thm:LMS_sparse_Ap_thm}]
To prove this theorem, we borrow some ideas in \cite[Theorem 3.2]{LMS}. However, we refine the argument in \cite[Theorem 3.2]{LMS} to provide a direct proof, and hence we avoid a duality argument for multilinear operators which may not be applicable in our setting.

Throughout the proof, let $\sigma_i=w_i^{1-p_i'}$, $\vec{f}\sigma =(f_1\sigma_1,\dotsc, f_m\sigma_m)$ and $f_i\geq 0$. Since we may assume that $w\in A_{\vec{P}}$, we have $\sigma_i, \nw \in A_\infty$ (see \cite[Theorem 3.6]{LOPTT}).

It suffices to prove that
\begin{equation}\label{eq1-thm1.1}
\|\mathcal{A}^\gamma_{\mathscr{D},\mathcal{S}}(\vec{f}\sigma)\|_{L^p(\nu_{\vec{w}})} \lesssim [\vec{w}]_{A_{\vec{P}}}^{\max(\f{1}{\gamma},\tfrac{p_1'}{p},\dotsc,\tfrac{p_m'}{p})} \prod_{i=1}^m \|f_i\|_{L^{p_i}(\sigma_i)}.
\end{equation}

Let $q=\min\{p,\gamma \}$. We get

$$
\begin{aligned}
\|\mathcal{A}^\gamma_{\mathscr{D},\mathcal{S}}(\vec{f}\sigma)\|_{L^p(\nu_{\vec{w}})}^q&= \left(\int_{\Rn} \Big[\sum_{Q\in \mathcal{S}} \Big(\prod_{i=1}^m \f{1}{|Q|}\int_Q f_i\sigma_i\Big)^\gamma \chi_{Q}(x)\Big]^{\frac{p}{\gamma}}\nw       \right)^{\frac{q}{p}}\\
 &\leq \left(\int_{\Rn} \Big[\sum_{Q\in \mathcal{S}} \Big(\prod_{i=1}^m \f{1}{|Q|}\int_Q f_i\sigma_i\Big)^q \chi_{Q}(x)\Big]^{\frac{p}{q}}\nw       \right)^{\frac{q}{p}},
\end{aligned}
$$
where we used the fact $q \leq \gamma$. Thus
\begin{equation}\label{eqdual-thm1.1}
\|\mathcal{A}^\gamma_{\mathscr{D},\mathcal{S}}(\vec{f}\sigma)\|_{L^p(\nu_{\vec{w}})}^q \leq \|[\mathcal{A}^q_{\mathscr{D},\mathcal{S}}(\vec{f}\sigma)]^q\|_{L^{p/q}(\nu_{\vec{w}})}.
\end{equation}

Denote $\beta=\max(\f{1}{q},\tfrac{p_1'}{p},\dotsc,\tfrac{p_m'}{p})$.  Also assume that $g\in L^{(p/q)'}(\nu_{\vec{w}})$ and $g\geq 0$.

 We have
$$
\int_{\mathbb{R}^n}[\mathcal{A}^q_{\mathscr{D},\mathcal{S}}(\vec{f}\sigma)]^q g\nu_{\vec{w}}=\sum_{Q\in \mathcal{S}}\int_Q g\nw \times \Big(\prod_{i=1}^m\f{1}{|Q|}\int_Q f_i\sigma_i\Big)^q.
$$
From this and the definition of $[\vec{w}]_{A_{\vec{P}}}$, we obtain
$$
\begin{aligned}
\sum_{Q\in \mathcal{S}}\int_Q g\nw& \times \Big(\prod_{i=1}^m\f{1}{|Q|}\int_Q f_i\sigma_i\Big)^q\\
&\leq [\vec{w}]_{A_{\vec{P}}}^{\beta q}\sum_{Q\in \mathcal{S}}\f{|Q|^{mq(\beta p-1)}}{\nw(Q)^{\beta q-1}\prod_{i=1}^m \sigma_i(Q)^{q(\beta p/p'_i-1)}}\times \Big(\f{1}{\nw(Q)}\int_Q g\nw\Big) \times \Big(\prod_{i=1}^m \f{1}{\sigma_i(Q)}\int_Q f_i\sigma_i\Big)^q\\
&\leq 2^{mq(\beta p-1)}[\vec{w}]_{A_{\vec{P}}}^{\beta q}\sum_{Q\in \mathcal{S}}\f{|E_Q|^{mq(\beta p-1)}}{\nw(E_Q)^{\beta q-1}\prod_{i=1}^m \sigma_i(E_Q)^{q(\beta p/p'_i-1)}}\times \Big(\f{1}{\nw(Q)}\int_Q g\nw\Big)\\
& \ \ \ \times \Big(\prod_{i=1}^m \f{1}{\sigma_i(Q)}\int_Q f_i\sigma_i\Big)^q
\end{aligned}
$$
where in the last inequality we used the facts $ \nw(Q) \geq \nw(E_Q)$, $\sigma_i(Q)\geq \sigma_i(E_Q)$ and  the positivity of the exponents. On the other hand, by H\"older's inequality, we have
\begin{equation}\label{eq2-thm1.1}
|E_Q|=\int_{E_Q}\nw^{\f{1}{mp}}\prod_{i=1}^m\sigma_i^{\f{1}{mp'_i}}\leq \nw(E_Q)^{\f{1}{mp}}\prod_{i=1}^m\sigma_i(E_Q)^{\f{1}{mp'_i}}.
\end{equation}
Insert this into the estimate above to conclude that
$$
\begin{aligned}
\sum_{Q\in \mathcal{S}}&\int_Q g\nw \times \Big(\prod_{i=1}^m\f{1}{|Q|}\int_Q f_i\sigma_i\Big)^q\\
&\leq 2^{mq(\beta p-1)}[\vec{w}]_{A_{\vec{P}}}^{\beta q}\sum_{Q\in \mathcal{S}} \left[ \Big(\f{1}{\nw(Q)}\int_Q g\nw\Big)\nw(E_Q)^{\f{1}{(p/q)^{'}}} \right]     \times \left[\prod_{i=1}^m \Big(\f{1}{\sigma_i(Q)}\int_Q f_i\sigma_i\Big) \sigma_i(E_Q)^{\f{1}{p_i}}          \right]^q
\end{aligned}
$$
which together with H\"older's inequality and the disjointness of the family $\{E_Q\}_{Q\in \mathcal{S}}$ gives
$$
\begin{aligned}
\sum_{Q\in \mathcal{S}}\int_Q g\nw \times \Big(\prod_{i=1}^m\f{1}{|Q|}\int_Q f_i\sigma_i\Big)^q
&\leq 2^{mq(\beta p-1)}[\vec{w}]_{A_{\vec{P}}}^{\beta q} \Big[\sum_{Q\in \mathcal{S}} \Big(\f{1}{\nw(Q)}\int_Q g\nw\Big)^{(p/q)'}\nw(E_Q)\Big]^{\f{1}{(p/q)'}}\\
 & \ \ \ \times \prod_{i=1}^m \Big[\sum_{Q\in \mathcal{S}}\Big(\f{1}{\sigma_i(Q)}\int_Q f_i\sigma_i\Big)^{p_i}\sigma_i(E_Q)\Big]^{q/p_i}\\
&\leq 2^{mq(\beta p-1)}[\vec{w}]_{A_{\vec{P}}}^{\beta q}\|M_{\nw}^{\mathscr{D}}(g)\|_{L^{(p/q)'}(\nw)} \times \prod_{i=1}^m \|M_{\sigma_i}^{\mathscr{D}}(f_i)\|_{L^{p_i}(\sigma_i)}^q\\
&\lesssim 2^{mq(\beta p-1)}[\vec{w}]_{A_{\vec{P}}}^{\beta q} \|g\|_{L^{(p/q)'}(\nw)} \times \prod_{i=1}^m\|f_i\|_{L^{p_i}(\sigma_i)}^q,
\end{aligned}
$$
{where to get the last inequality we applied (\ref{bbs})}. Hence,

$$
\begin{aligned}
\|\mathcal{A}^\gamma_{\mathscr{D},\mathcal{S}}(\vec{f}\sigma)\|^q_{L^{p}(\nw)} &\overset{\text{(\ref{eqdual-thm1.1})}}{\leq}\ \|[\mathcal{A}^q_{\mathscr{D},\mathcal{S}}(\vec{f}\sigma)]^q\|_{L^{p/q}(\nu_{\vec{w}})} \\
&\leq \sup_{\| g\|_{L^{(p/q)'}(\nu_{\vec{w}})}=1 } \int_{\mathbb{R}^n}[\mathcal{A}^q_{\mathscr{D},\mathcal{S}}(\vec{f}\sigma)]^q g\nu_{\vec{w}}\\
&\leq 2^{mq(\beta p-1)}[\vec{w}]_{A_{\vec{P}}}^{\beta q} \times \prod_{i=1}^m\|f_i\|_{L^{p_i}(\sigma_i)}^q.
\end{aligned}
$$

This proves \eqref{eq1-thm1.1}.
\end{proof}

 In order to prove Theorem \ref{thm:sharp}, we use the approach of \cite{Ler5}. Let $\Phi$ be a fixed Schwartz function such that
$$
\chi_{B(0,1)}(x)\leq \Phi(x)\leq \chi_{B(0,2)}(x).
$$
We define
$$
\begin{aligned}
\widetilde{S}_{\alpha, \psi}(\vec{f})(x)&=\Big(\int_{\mathbb{R}^{n+1}_+}\Phi\Big(\f{x-y}{t\alpha}\Big)|\psi_t(\vec{f})(y)|^2\f{dydt}{t^{n+1}}\Big)^{1/2}.
\end{aligned}
$$
It  easy to see that
\begin{equation}\label{cutoffsquarefunction}
S_{\alpha, \psi}(\vec{f})(x)\leq \widetilde{S}_{\alpha, \psi}(\vec{f})(x)\leq S_{2\alpha, \psi}(\vec{f})(x).
\end{equation}

As a generalization of  \cite[Lem. 3.1]{Ler5} for multilinear case, we have
  \begin{prop}\label{prop:osc}
   For any cube $Q\subset\Rn$, $\alpha\geq 1$ and $\delta_0<\min\{\delta,1/2\}$, we have
    \begin{equation*}
      \omega_{\lambda}(\widetilde{S}_{\alpha, \psi}(\vec{f})^2;Q)\leq c_{m,n,\lambda,\psi} \alpha^{2mn} \sum_{l=0}^\infty \frac{1}{2^{l\delta_0}} \Big(\prod_{i=1}^m \frac{1}{|2^l Q|} \int_{2^l Q} |f_i(y)| dy\Big)^2.
    \end{equation*}
  \end{prop}

\begin{proof}[Proof of Proposition~\ref{prop:osc}]
Without the loss of generality we may assume that $\delta<1/2$.

For a cube $Q\subset \mathbb{R}^n$ we set $T(Q)=Q\times (0,\ell(Q))$. We then write
$$
\begin{aligned}
\widetilde{S}_{\alpha, \psi}(\vec{f})^2(x)&= \int_{T(2Q)}\Phi\Big(\f{x-y}{\alpha t}\Big)|\psi_t(\vec{f})(y)|^2\f{dydt}{t^{n+1}}
+\int_{\mathbb{R}^{n+1}_+\backslash T(2Q)}\Phi\Big(\f{x-y}{\alpha t}\Big)|\psi_t(\vec{f})(y)|^2\f{dydt}{t^{n+1}}\\
&=E(\vec{f})(x)+F(\vec{f})(x).
\end{aligned}
$$
We set $\vec{f^0}=(f_1\chi_{Q^*},\dotsc, f_m\chi_{Q^*})$, where $Q^*=8Q$. For each $i=1,\dots,m$, we set $f_i^0=f_i\chi_{Q^*}$ and $f_i^\vc=f_i\chi_{(Q^*)^c}$. Then we have
    \begin{equation}\label{eq1}
    E(\vec f)(z) \leq 2^m E(\vec {f}^0)(z)+2^m \sum_{\vec{\alpha}\in \mathcal{I}_0}  E\Big[(f_1^{\a_1},\ldots,f_m^{\a_m})\Big](z),
  \end{equation}
where $\mathcal{I}_0:=\{\vec{\alpha}=(\alpha_1,\dots,\alpha_m): \, \alpha_i\in\{0,\vc\}, \ \text{ and at least one $\alpha_i\neq 0$}\}$. We denote the vector $\vec{\alpha}$ by $\vec{0}$ if $\alpha_i=0$ for all $ 1\leq i\leq m$. Therefore,
$$
(E(\vec{f})\chi_Q)^*(\lambda|Q|)\leq 2^m\Big\{(E(\vec{f^0})\chi_Q)^*(\lambda|Q|/2^m)+\sum_{\vec{\alpha}\in \mathcal{I}_0}\Big[E(f_1^{\a_1},\ldots,f_m^{\a_m})\chi_Q\Big]^*(\lambda|Q|/2^m)\Big\}.
$$
Due to \eqref{cutoffsquarefunction} and Lemma \ref{weakestimate}, $\|\widetilde{S}_{\alpha, \psi}(\vec{f})\|_{L^{1/m,\vc}}\leq c_{m,n} \alpha^{mn}\|S_{1, \psi}(\vec{f})\|_{L^{1/m,\vc}}$. This together with the fact that $S_{1, \psi}$ maps continuously from $L^1\times \ldots \times L^1$ into $L^{1/m,\vc}$ yields that
$$
\begin{aligned}
(E(\vec{f^0})\chi_Q)^*(\lambda|Q|/2^m)&\leq (\widetilde{S}_{\alpha, \psi}(\vec{f^0})\chi_Q)^*(\lambda|Q|/2^m)^2\\
&\leq c_{n,m,\lambda,\psi} \alpha^{2mn} \Big(\prod_{j=1}^m \f{1}{|Q^*|}\int_{Q^*}|f_j|\Big)^2.
\end{aligned}
$$
On the other hand, for each $\vec{\a}\in \mathcal{I}_0$ we have
$$
\Big[E(f_1^{\a_1},\ldots,f_m^{\a_m})\chi_Q\Big]^*(\lambda|Q|/2^m)
\leq \f{2^m}{\lambda|Q|}\int_{\mathbb{R}^n}\int_{T(2Q)}\Phi\Big(\f{x-y}{\alpha t}\Big)\Big|\psi_t(f_1^{\a_1},\ldots ,f_m^{\a_m})(y)\Big|^2\f{dydt}{t^{n+1}}dx.
$$
This along with the fact that
$$
\int_{\mathbb{R}^n}\Phi\Big(\f{x-y}{\alpha t}\Big)dx\leq c_n(\alpha t)^n
$$
implies that
$$
\Big[E(f_1^{\a_1},\ldots,f_m^{\a_m})\chi_Q\Big]^*(\lambda|Q|/2^m)\leq c_{n}\f{2^m}{\lambda|Q|}\int_{T(2Q)}(\alpha t)^n|\psi_t(f_1^{\a_1},\ldots ,f_m^{\a_m})(y)|^2\f{dydt}{t^{n+1}}.
$$
Hence for $y\in 2Q$ and $(\alpha_1,\dots,\alpha_m)\in \mathcal{I}_0$, by \eqref{eq-sizecondition},
$$
\begin{aligned}
|\psi_t(f_1^{\a_1},\ldots ,f_m^{\a_m})(y)|
&\leq A \int_{(\mathbb{R}^n)^m}\f{t^\delta}{(t+|y-z_1|+\dotsb+|y-z_m|)^{mn+\delta}}\prod_{j=1}^m|f^{\alpha_j}_j(z_j)|d(z_j)\\
&\leq A \int_{(\mathbb{R}^n)^m}\f{t^\delta}{(|y-z_1|+\dotsb+|y-z_m|)^{mn+\delta}}\prod_{j=1}^m|f^{\alpha_j}_j(z_j)|d(z_j)\\
&\leq A (t/\ell(Q))^\delta\int_{(\mathbb{R}^n)^m}\f{\ell(Q)^\delta}{(|y-z_1|+\dotsb+|y-z_m|)^{mn+\delta}}\prod_{j=1}^m|f^{\alpha_j}_j(z_j)|d(z_j)\\
&\leq A (t/\ell(Q))^\delta\Big[\int_{(8Q)^m}\dots+ \sum_{k\geq 3}\int_{(2^{k+1}Q)^m\backslash (2^{k}Q)^m}\dots \Big]\\
&\leq c_{n} (t/\ell(Q))^\delta \sum_{k=0}^\vc \f{1}{2^{k\delta}}\Big(\prod_{j=1}^m \f{1}{|2^{k}Q|}\int_{2^{k} Q}|f_j|\Big).
\end{aligned}
$$
 These two estimates give that for $\vec{\a}\in \mathcal{I}_0$
$$
\begin{aligned}
\Big[E(f_1^{\a_1},\ldots,f_m^{\a_m})&\chi_Q\Big]^*(\lambda|Q|/2^m)\\
&\leq c_n \Big[\sum_{l=0}^\vc \f{1}{2^{l\delta}}\Big(\prod_{j=1}^m \f{1}{|2^lQ|}\int_{2^l Q}|f_j|\Big)\Big]^2\f{2^m}{\lambda|Q|}\int_{T(2Q)}(\alpha t)^n (t/\ell(Q))^{2\delta} \f{dydt}{t^{n+1}}\\
&\leq c_{n,\lambda,\psi} \alpha^n \Big[\sum_{l=0}^\vc \f{1}{2^{l\delta}}\Big(\prod_{j=1}^m \f{1}{|2^lQ|}\int_{2^l Q}|f_j|\Big)\Big]^2\\
&\leq c_{n,\lambda,\psi} \alpha^n \sum_{l=0}^\vc \f{1}{2^{l\delta}}\Big(\prod_{j=1}^m \f{1}{|2^lQ|}\int_{2^l Q}|f_j|\Big)^2\\
\end{aligned}
$$
where in the last inequality we used H\"older's inequality.

Therefore,
$$
(E(\vec{f})\chi_Q)^*(\lambda|Q|)\leq c_{n,m,\lambda,\psi} \alpha^{2mn} \sum_{l=0}^\vc \f{1}{2^{l\delta}}\Big(\prod_{j=1}^m \f{1}{|2^lQ|}\int_{2^l Q}|f_j|\Big)^2.
$$

To complete the proof, we will claim that
\begin{equation}\label{estimateonF}
|F(\vec{f})(x)-F(\vec{f})(x_Q)|\leq c_{n,\lambda,\psi} \alpha^{2mn} \sum_{l=0}^\vc \f{1}{2^{l\delta}}\Big(\prod_{j=1}^m \f{1}{|2^lQ|}\int_{2^l Q}|f_j|\Big)^2,
\end{equation}
for all $x\in Q$, where $x_Q$ is the center of $Q$.

Once we can prove \eqref{estimateonF}, the conclusion of the proposition follows immediately by using the fact that
$$
\omega_{\lambda}(\widetilde{S}_{\alpha, \psi}(\vec{f})^2;Q)\leq (E(\vec{f})\chi_Q)^*(\lambda|Q|)+ \|F(\vec{f})-F(\vec{f})(x_Q)\|_{L^\vc(Q)}.
$$

We now prove \eqref{estimateonF}. We first write
$$
|F(\vec{f})(x)-F(\vec{f})(x_Q)|\leq \sum_{l=1}^\vc \int_{T(2^{l+1}Q)\backslash T(2^l Q)}\Big|\Phi\Big(\f{x-y}{\alpha t}\Big)-\Phi\Big(\f{x_Q-y}{\alpha t}\Big)\Big| \, |\psi_t(\vec{f})(y)|^2\f{dydt}{t^{n+1}}.
$$
Note that if $t<\f{2^l -1}{4\alpha}\ell(Q)$ then $\min\{|x-y|, |x_Q-y|\}>2\alpha t$ for all $(y,t)\in T(2^{l+1}Q)\backslash T(2^l Q)$ and $x\in Q$. Hence,
$$
\Phi\Big(\f{x-y}{\alpha t}\Big)-\Phi\Big(\f{x_Q-y}{\alpha t}\Big)=0.
$$
As a consequence, we have
$$
\begin{aligned}
|F(\vec{f})(x)&-F(\vec{f})(x_Q)|\\
&\leq \sum_{l=1}^\vc \int_{T(2^{l+1}Q)\backslash T(2^l Q)}\Big|\Phi\Big(\f{x-y}{\alpha t}\Big)-\Phi\Big(\f{x_Q-y}{\alpha t}\Big)\Big|\, |\psi_t(\vec{f})(y)|^2\chi_{[\f{2^l -1}{4\alpha}\ell(Q),2^{l+1}\ell(Q))}(t)\f{dydt}{t^{n+1}}\\
&\leq \sum_{l=1}^\vc \int_{T(2^{l+1}Q)\backslash T(2^l Q)}\Big|\Phi\Big(\f{x-y}{\alpha t}\Big)-\Phi\Big(\f{x_Q-y}{\alpha t}\Big)\Big|\, |\psi_t(\vec{f})(y)|^2\chi_{[\f{2^{l-3}}{\alpha}\ell(Q),2^{l+1}\ell(Q))}(t)\f{dydt}{t^{n+1}}.
\end{aligned}
$$
It is easy to see that for $x\in Q$ we have
$$
\Big|\Phi\Big(\f{x-y}{\alpha t}\Big)-\Phi\Big(\f{x_Q-y}{\alpha t}\Big)\Big|\leq c_{n,\Phi}\f{|x-x_Q|}{\alpha t}\leq c_{n,\Phi}\f{\ell(Q)}{\alpha t}.
$$
Now we set  $\vec{f}^0=(f_1\chi_{Q_l},\dotsc, f_m\chi_{Q_l})$, where $Q_l=2^{l+2}Q$. For each $i=1,\dots,m$, we set $f_i^0=f_i\chi_{Q_l}$ and $f_i^\vc=f_i\chi_{(Q_l)^c}$. Denote
$$
F^{\vec{\alpha}}(\vec{f})= \sum_{l=1}^\vc (\ell(Q)/\alpha)\int_{2^{l+1}Q}\int_{\f{2^{l-3}}{\alpha}\ell(Q)}^{2^{l+1}\ell(Q)}|\psi_t(f_1^{\a_1},\ldots,f_m^{\a_m})(y)|^2\f{dydt}{t^{n+2}}.
$$
Therefore,
$$
\begin{aligned}
|F(\vec{f})(x)-F(\vec{f})(x_Q)|&\leq \sum_{l=1}^\vc (\ell(Q)/\alpha)\int_{2^{l+1}Q}\int_{\f{2^{l-3}}{\alpha}\ell(Q)}^{2^{l+1}\ell(Q)}|\psi_t(\vec{f})(y)|^2\f{dydt}{t^{n+2}}\\
&\leq 2^m F^{\vec{0}}(\vec{f})(x)+  2^m \sum_{\vec{\alpha}\in \mathcal{I}_0} F^{\vec{\alpha}}(\vec{f})  (x).
\end{aligned}
$$
For the first term, using \eqref{eq-sizecondition} to get that
$$
\begin{aligned}
F&^{\vec{0}}(\vec{f})(x)\\
&\leq A\sum_{l=1}^\vc (\ell(Q)/\alpha)\int_{2^{l+1}Q}\int_{\f{2^{l-3}}{\alpha}\ell(Q)}^{2^{l+1}\ell(Q)}\Big|\int_{(2^{l+2}Q)^m}\f{t^\delta}{(t+|y-z_1|+\dotsb+|y-z_m|)^{mn+\delta}}\prod_{j=1}^m|f_j(z_j)|dz_j\Big|^2\f{dydt}{t^{n+2}},
\end{aligned}
$$
which along with the fact that
$$
\begin{aligned}
\int_{2^{l+1}Q}&\Big|\f{t^\delta}{(t+|y-z_1|+\dotsb+|y-z_m|)^{mn+\delta}}\Big|^2dy\\
&=\f{1}{t^{2mn-n}}\int_{2^{l+1}Q}\f{1}{t^n}\left[\f{t}{(t+|y-z_1|+\dotsb+|y-z_m|)}\right]^{2mn+2\delta}dy\\
&\leq\f{1}{t^{2mn-n}}\int_{2^{l+1}Q}\f{1}{t^n}\left(\f{t}{t+|y-z_1|}\right)^{2mn+2\delta}dy\\
&\leq\f{1}{t^{2mn-n}}\int_{\mathbb{R}^n}\f{1}{t^n}\left(\f{t}{t+|y-z_1|}\right)^{n+\delta}dy\\
&\leq \f{c_n}{t^{2mn-n}}
\end{aligned}
$$
and Minkowski's inequality implies that
$$
\begin{aligned}
F^{\vec{0}}(\vec{f})(x)&\leq c_n \sum_{l=1}^\vc (\ell(Q)/\alpha)\Big[\int_{(2^{l+2}Q)^m}\Big(\int_{\f{2^{l-3}}{\alpha}\ell(Q)}^{2^{l+1}\ell(Q)}\f{dt}{t^{2mn+2}}\Big)^{1/2}\prod_{j=1}^m|f_j(z_j)|dz_j\Big]^{2}\\
&\leq c_n \sum_{l=1}^\vc (\ell(Q)/\alpha)\Big[\int_{(2^{l+2}Q)^m}\Big(\int_{\f{2^{l-3}}{\alpha}\ell(Q)}^{\vc}\f{dt}{t^{2mn+2}}\Big)^{1/2}\prod_{j=1}^m|f_j(z_j)|dz_j\Big]^{2}\\
&\leq c_n \sum_{l=1}^\vc (\ell(Q)/\alpha)\Big[\int_{(2^{l+2}Q)^m}\Big(\f{\alpha}{2^{l-3}\ell(Q)}\Big)^{mn+1/2}\prod_{j=1}^m|f_j(z_j)|dz_j\Big]^{2}\\
&\leq c_n \alpha^{2mn}\sum_{k=1}^\vc \f{1}{2^{k}}\Big(\prod_{j=1}^m \f{1}{|2^kQ|}\int_{2^k Q}|f_j|\Big)^2\\
&\leq c_n \alpha^{2mn}\sum_{k=1}^\vc \f{1}{2^{k\delta}}\Big(\prod_{j=1}^m \f{1}{|2^kQ|}\int_{2^k Q}|f_j|\Big)^2.
\end{aligned}
$$
For the second term $\sum_{\vec{\alpha}\in \mathcal{I}_0} F^{\vec{\alpha}}(\vec{f})(x)$, similar to previous computation, using \eqref{eq-sizecondition} we get that, for $(\alpha_1,\dots,\alpha_m)\in \mathcal{I}_0$ and $(y,t)\in T(2^{l+1}Q)$ ,
$$
\begin{aligned}
|\psi_t(f_1^{\a_1},\ldots ,f_m^{\a_m})(y)|
&\leq A \int_{(\mathbb{R}^n)^m}\f{t^\delta}{(t+|y-z_1|+\dotsb+|y-z_m|)^{mn+\delta}}\prod_{j=1}^m|f^{\alpha_j}_j(z_j)|d(z_j)\\
&\leq A \Big[\int_{(2^{l+1}Q)^m}\dots+\sum_{k\geq 1}\int_{(2^{l+k+1}Q)^m\backslash (2^{l+k}Q)^m}\dots\Big]\\
&\leq c_{n,\psi} (t/\ell(Q))^\delta \sum_{k=0}^\vc \f{1}{2^{(k+l)\delta}}\Big(\prod_{j=1}^m \f{1}{|2^{k+l}Q|}\int_{2^{k+l} Q}|f_j|\Big)\\
&\leq c_{n,\psi} (t/\ell(Q))^\delta \sum_{k=0}^\vc \f{1}{2^{k\delta}}\Big(\prod_{j=1}^m \f{1}{|2^{k}Q|}\int_{2^{k} Q}|f_j|\Big).
\end{aligned}
$$
{Plugging} this estimate into the expression of $F^{\vec{\alpha}}(\vec{f})(x)$ and by a straightforward calculation we obtain
$$
\begin{aligned}
\sum_{\vec{\alpha}\in \mathcal{I}_0} F^{\vec{\alpha}}(\vec{f})(x) &\leq c_{n,\psi}\alpha^{n-2\delta}\sum_{l=1}^\vc \f{2^{2l\delta}}{2^{l}} \Big(\sum_{k=l}^\vc \f{1}{2^{k\delta}}\prod_{j=1}^m \f{1}{|2^{k}Q|}\int_{2^{k} Q}|f_j|\Big)^2\\
&\leq c_{n,\psi}\alpha^{n-2\delta} \sum_{k=1}^\vc \f{1}{2^{k\delta}}\Big(\prod_{j=1}^m \f{1}{|2^{k}Q|}\int_{2^k Q}|f_j|\Big)^2
\end{aligned}
$$
provided $\delta<1/2$.

This completes our proof.
\end{proof}

The conclusion in Theorem \ref{thm:sharp} follows immediately from Theorem \ref{Thm:LMS_sparse_Ap_thm} and the following result.
\begin{prop}\label{prop3.3}
Let $w$ be a weight, $0<p<\vc$ and $\alpha \geq 1$. Then for any appropriate $\vec{f}$, we have
$$
\|S_{\alpha,\psi}(\vec{f})\|_{L^p(w)}\leq c(m,n,\psi)\alpha^{mn}\sup_{\mathscr{D},\mathcal{S}}\|\mathcal{A}^2_{\mathscr{D},\mathcal{S}}(|\vec{f}|)\|_{L^p(w)}.
$$
\end{prop}
\begin{proof}
From Theorem \ref{decom1} and Proposition \ref{prop:osc}, for $Q_0\in \mathcal{D}$, we can pick a sparse family in $Q_0$ which is denoted by $\mathcal{S}=\mathcal{S}(Q_0):=\{Q\}\in \mathcal{D}$ so that
$$
\begin{aligned}
|S_{\alpha,\psi}(\vec{f})(x)^2&-m_{S_{\alpha,\psi}(\vec{f})^2}(Q_0)|\\
&\leq c_{n,m,\psi}\alpha^{2mn}\Big\{\sum_{Q\in \mathcal{S}(Q_0)}\sum_{l=0}^\vc 2^{-l\delta}\Big(\prod_{i=1}^m\f{1}{|2^l Q|}\int_{2^l Q}|f_i(y)|dy\Big)^2\chi_{Q}(x)\Big\}\\
&\leq c_{n,m,\psi}\alpha^{2mn}\Big\{ \sum_{l=0}^\vc 2^{-l\delta} (\mathcal{T}^2_{\mathcal{S}(Q_0),l}(\vec{f})(x))^2\Big\},
\end{aligned}
$$
for a.e. $x\in Q_0$, where
$$
\mathcal{T}^\gamma_{\mathcal{S},l}(\vec{f})(x)=\Big[\sum_{Q\in \mathcal{S}}\Big(\prod_{i=1}^m\f{1}{|2^l Q|}\int_{2^l Q}|f_i(y)|dy\Big)^\gamma\chi_{Q}(x)\Big]^{1/\gamma},
$$
for $\gamma\geq 1$  and sparse family $\mathcal{S}$ in $\mathcal{D}$.

We now observe that
$$
\mathcal{T}^2_{\mathcal{S}(Q_0),\ell}\vec{f} = \Big[\mathcal{T}^1_{\mathcal{S}(Q_0),\ell}(\vec{f},\vec{f})\Big]^{1/2}.
$$
On the other hand, the argument in Sections 11-13 in \cite{LN} shows that
$$
\begin{aligned}
\sum_{l=0}^\vc 2^{-l\delta} \mathcal{T}^1_{\mathcal{S}(Q_0),l}(\vec{f},\vec{f})(x)&\leq c_{m,n,\delta} \sup_{\mathscr{D},\mathcal{S}}\mathcal{A}^1_{\mathscr{D},\mathcal{S}}(|\vec{f}|, |\vec{f}|)(x)\\
&\leq c_{m,n,\delta}\sup_{\mathscr{D},\mathcal{S}}[\mathcal{A}^2_{\mathscr{D},\mathcal{S}}(|\vec{f}|)(x)]^2.
\end{aligned}
$$

Hence, we obtain that
\begin{equation}\label{eq5-thm1.1}
\begin{aligned}
|S_{\alpha,\psi}(\vec{f})(x)^2-m_{S_{\alpha,\psi}(\vec{f})^2}(Q_0)|\leq c_{n,m,\psi}\alpha^{2mn} \sup_{\mathscr{D},\mathcal{S}}[\mathcal{A}^2_{\mathscr{D},\mathcal{S}}(|\vec{f}|)(x)]^2,
\end{aligned}
\end{equation}
for a.e. $x\in Q_0$.

Since $S_{\alpha,\psi}$ maps $L^{1}\times\dotsb\times L^{1}$ into $L^{1/m,\vc}$, $\lim_{|Q_0|\to \vc}m_{S_{\alpha,\psi}(\vec{f})^2}(Q_0)=0$ provided $\vec{f}\in L^{1}\times \ldots\times L^1$. This together with \eqref{eq5-thm1.1} implies that
\begin{equation}\label{eq6-thm1.1}
S_{\alpha,\psi}(\vec{f})(x)^2\leq c_{n,m,\psi}\alpha^{2mn} \sup_{\mathscr{D},\mathcal{S}}[\mathcal{A}^2_{\mathscr{D},\mathcal{S}}(|\vec{f}|)(x)]^2.
\end{equation}
Hence,
$$
\|S_{\alpha,\psi}(\vec{f})\|_{L^p(w)}\leq c_{m,n,\psi}\alpha^{mn}\sup_{\mathcal{S}\in \mathcal{D}}\|\mathcal{A}^2_{\mathscr{D},\mathcal{S}}(|\vec{f}|)\|_{L^p(w)}.
$$
This completes our proof.
\end{proof}

\bigskip

\begin{proof}[Proof of Theorem \ref{thm:sharp2}:]
We first observe that
$$
g^*_{\lambda,\psi}(\vec{f})(x)^2\leq \sum_{k=1}^\vc 2^{-kn\lambda}S_{2^k,\psi}(\vec{f})(x)^2,
$$
which together with \eqref{eq6-thm1.1} implies that
$$
\begin{aligned}
g^*_{\lambda,\psi}(\vec{f})(x)^2&\leq c_{n,m,\psi}\sum_{k=1}^\vc 2^{-kn\lambda}2^{2kmn}\Big\{ \sup_{\mathscr{D},\mathcal{S}}[\mathcal{A}^2_{\mathscr{D},\mathcal{S}}(|\vec{f}|)(x)]^2\Big\}\\
&\leq c_{n,m,\psi}\Big\{ \sup_{\mathscr{D},\mathcal{S}}[\mathcal{A}^2_{\mathscr{D},\mathcal{S}}(|\vec{f}|)(x)]^2\Big\},
\end{aligned}
$$
provided $\lambda>2m$.

This implies that for $\vec{w}\in A_{\vec{P}}$ and $p>0$ we have
$$
\|g^*_{\lambda,\psi}(\vec{f})\|_{L^p(\nw)}\leq c_{n,m,\psi} \sup_{\mathcal{S}\in \mathcal{D}}\|\mathcal{A}^2_{\mathscr{D},\mathcal{S}}(|\vec{f}|)\|_{L^p(w)}.
$$
The conclusion in Theorem \ref{thm:sharp2} follows immediately from Theorem \ref{Thm:LMS_sparse_Ap_thm}.
\end{proof}

\end{document}